\documentclass[
final
]{dmtcs-episciences}

\usepackage[utf8]{inputenc}
\usepackage{subfigure}

\usepackage[round]{natbib}

\usepackage{
	amsmath,			
	amssymb,			
	enumerate,		    
	graphicx,			
	lastpage,			
	multicol,			
	multirow,			
	pifont,			    
}
\usepackage{amsfonts}
\usepackage{amsthm}
\usepackage{mathtools}
\DeclareMathOperator{\tw}{\mathrm{tw}}
\DeclareMathOperator{\diam}{\mathrm{diam}}
\DeclareMathOperator{\inv}{\mathrm{inv}}
\DeclareMathOperator{\dist}{\mathrm{dist}}

\newtheorem{theorem}{Theorem}[section]
\newtheorem{definition}[theorem]{Definition}
\newtheorem{lemma}[theorem]{Lemma}
\newtheorem{claim}[theorem]{Claim}

\newtheorem{proposition}[theorem]{Proposition}
\newtheorem{conjecture}[theorem]{Conjecture}

\author[Yichen Wang et al.]{Yichen Wang\affiliationmark{1}
  \and Haozhe Wang\affiliationmark{1}
  \and Yuxuan Yang\affiliationmark{2}\thanks{Corresponding Author}
  \and Mei Lu\affiliationmark{1}}

\title{Inversion diameter and treewidth}

\affiliation{
  Department of Mathematical Sciences, Tsinghua University, Beijing, China\\
  School of Mathematical Sciences, Beijing University of Posts and Telecommunications, Beijing, China}
\begin{document}
\publicationdata{vol. 28:2}{2026}{29}{10.46298/dmtcs.16074}{2025-07-20; 2025-07-20; 2025-11-26; 2026-02-27}{2026-05-01}

\maketitle
\begin{abstract}
  In an oriented graph $\overrightarrow{G}$, the inversion of a subset $X$ of vertices is the operation that reverses the orientation of all arcs with both end-vertices in $X$. The inversion graph of a graph $G$, denoted by $\mathcal{I}(G)$, is the graph whose vertices are orientations of $G$ in which two orientations $\overrightarrow{G_1}$ and $\overrightarrow{G_2}$ are adjacent if and only if there is an inversion transforming $\overrightarrow{G_1}$ into $\overrightarrow{G_2}$.
The inversion diameter of a graph $G$ is the diameter of its inversion graph $\mathcal{I}(G)$, denoted by $\mathrm{diam}(\mathcal{I}(G))$.
Havet, H\"orsch, and Rambaud~(2024) first proved that for $G$ of treewidth $k$, $\mathrm{diam}(\mathcal{I}(G)) \le 2k$, and that there are graphs of treewidth $k$ with inversion diameter $k+2$.
In this paper, we construct graphs of treewidth $k$ with inversion diameter $2k$, which implies that the previous upper bound $\mathrm{diam}(\mathcal{I}(G)) \le 2k$ is tight.
Moreover, for graphs with maximum degree $\Delta$, Havet, H\"orsch, and Rambaud~(2024) proved $\mathrm{diam}(\mathcal{I}(G)) \le 2\Delta-1$ and conjectured that $\mathrm{diam}(\mathcal{I}(G)) \le \Delta$. We prove the conjecture when $\Delta=3$ with the help of computer calculations.
\end{abstract}

\section{Introduction}
An \emph{orientation} of an undirected graph is an assignment of a direction to each edge, turning the initial graph into a directed graph. Let $G$ be a simple graph and $\overrightarrow{G_1}$ an orientation of $G$. If $X$ is a vertex subset of $G$, the \emph{inversion} of $X$ on $\overrightarrow{G_1}$ is the operation that reverses the direction of all arcs with both end-vertices in $X$, and results in a new
orientation $\overrightarrow{G_2}$.

The concept of inversion was first introduced by~\cite{belkhechine2010inversion}. They studied the \emph{inversion number} of a directed graph $D$, denoted by $\inv(D)$, which is the minimum number of inversions that transform $D$ into an acyclic graph. They proved, for every fixed $k$, given a tournament $T$, determining whether $\inv(T) \le k$ is polynomial-time solvable.
In contrast,~\cite{bang2022inversion} proved that given any directed graph $D$, determining whether $\inv(D) \le 1$ is NP-complete. 

The maximum inversion number across all oriented graphs of order $n$, denoted by $\inv(n)$, has also been investigated.~\cite{aubian2025problems} and~\cite{alon2024invertibility} proved $n-2\sqrt{2\log n} \le \inv(n) \le n - \lceil \log (n+1) \rceil$. Besides these results, various related questions have also been studied.

Let $G$ be a simple graph. An inversion is a transformation between different orientations of $G$. Instead of transforming an orientation into an acyclic orientation, it is also natural to consider the inversion between two orientations. The \emph{inversion graph} of $G$ denoted by $\mathcal{I}(G)$, is the graph whose vertices are the orientations of $G$ in which two orientations $\overrightarrow{G_1}$ and $\overrightarrow{G_2}$ are adjacent if and only if there is an inversion $X$ transforming $\overrightarrow{G_1}$ into $\overrightarrow{G_2}$. The \emph{inversion diameter} of $G$ is the diameter of $\mathcal{I}(G)$, denoted by $\diam(\mathcal{I}(G))$. It represents the maximum number of inversions required to transform an orientation of $G$ into another.

\cite{havet2024diameter} first introduced inversion diameter and studied its behaviour on various classes of graphs.
Let $G$ be a graph and let $<$ be a total ordering on $V(G)$. For every pair $u, u'$ of vertices in $G$, let $N_{<u'}(u) = \{ v\in N(u) \mid v<u' \}$ and $N_{>u'}(u) = \{ v \in N(u) \mid v>u' \}$. We simply write $N_<(u)$ for $N_{<u}(u)$ and $N_>(u)$ for $N_{>u}(u)$.
The ordering $<$ is \emph{$t$-strong} if for every $u \in V(G)$
\begin{itemize}
    \item $|N_<(u)| + \log( |\{X \subseteq V(G) \mid \exists v \in N_>(u) , X \subseteq N_{<u}(v) \}| ) < t$, if $N_>(u) \neq \emptyset$, and
    \item $N_{<}(u) \le t $ otherwise.
\end{itemize}

A graph is \emph{strongly $t$-degenerate} if it admits a $t$-strong ordering of its vertices.
~\cite{havet2024diameter} showed that
\begin{theorem}[\cite{havet2024diameter}]\label{thm: strong t degenerate}
Let $G$ be a graph and let $t$ be a positive integer.
If $G$ is strongly $t$-degenerate, then $\diam(\mathcal{I}(G)) \le t$.
\end{theorem}

As corollaries of Theorem~\ref{thm: strong t degenerate}, they showed that various properties of a graph can be used to bound the diameter of its inversion graph.


\begin{theorem}[~\cite{havet2024diameter}]\label{thm: havet}

\par~
\begin{enumerate}
    \item For every graph $G$ with at least one edge and maximum degree $\Delta$, $\diam(\mathcal{I}(G)) \le 2\Delta - 1$.
    \item $\diam(\mathcal{I}(G)) \le 12$ for every planar graph $G$.
    \item
    $\diam(\mathcal{I}(G)) \le 2k$ for every graph $G$ of treewidth at most $k$.
\end{enumerate}
\end{theorem}

\cite{havet2024diameter} also proved that for given $k \ge 2$ and a graph $G$, determining whether $\diam(\mathcal{I}(G)) \le k$ is NP-hard.
For a graph $G$ with maximum degree $3$~(a sub-cubic graph),~\cite{havet2024diameter} showed a better bound $\diam(\mathcal{I}(G)) \le 4$. 
Moreover, they proposed the following conjecture on graphs with maximum degree $\Delta$.

\begin{conjecture}[\cite{havet2024diameter}]\label{conj: Delta}
For every graph $G$ with at least one edge and maximum degree $\Delta$, $\diam(\mathcal{I}(G)) \le \Delta$.
\end{conjecture}

The conjecture is true for $\Delta \le 2$~\cite{havet2024diameter}.
In this paper, we prove the conjecture when $\Delta = 3$.
Computer assistance will be used in the proof of Theorem~\ref{thm: Delta 3 diam}.
A pure mathematical proof is still worth studying.

\begin{theorem}\label{thm: Delta 3 diam}
If $G$ is a graph of maximum degree $3$, then $\diam(\mathcal{I}(G)) \le 3$.
\end{theorem}


For graphs with treewidth at most $k$,~\cite{havet2024diameter} showed that there are graphs of treewidth at most $k$ with inversion diameter $k+2$.
In this paper, we show that the upper bound $\diam(\mathcal{I}(G)) \le 2k$ for graphs of treewidth at most $k$ is tight by proving Theorem~\ref{thm: main}. In doing so, we answer a question proposed by ~\cite{havet2024diameter}.

\begin{theorem}\label{thm: main}
    For every positive integer $k$, there are graphs of treewidth $k$ with inversion diameter $2k$.
\end{theorem}

This paper is organized as follows.
In Section~\ref{sec: preliminary}, we give basic definitions and notation. The proofs of Theorems~\ref{thm: main} and~\ref{thm: Delta 3 diam} are given
in Sections~\ref{sec: main} and~\ref{sec: delta 3}, respectively.

\section{Preliminaries}\label{sec: preliminary}

Let $G=(V,E)$ be a graph. The \emph{distance} between $u$ and $v$, denoted by $\dist(u,v)$, is the number of edges in a shortest path joining $u$ and $v$. For any vertex $u\in V(G)$, let $N(u)=\{v \mid uv\in E(G)\}$ and denote by $d(u)=|N(u)|$ the degree of $u$. Let $\Delta(G)$  be the maximum degree of $G$. We call $G$ \emph{$k$-regular} if $d(u)=k$ for every $u\in V(G)$. Let $G$ be a graph and $S$ a vertex subset of its vertices. 
Let $G[S]$ denote the subgraph of $G$ induced by $S$.
For a graph $G$ and a vertex $v$, denote by $G-v$  the graph induced by $V(G) - \{v\}$.
For a graph $G$ and an induced subgraph $H$, denote by $G-H$  the graph induced by $V(G) - V(H)$.

A \emph{labelling} of $G$ is a mapping $\pi: E(G) \rightarrow \mathbb{F}_2$.
A \emph{$t$-dim vector assignment} of $G$ respecting the labelling $\pi$ is a mapping $f: V(G) \rightarrow \mathbb{F}_2^t$ such that $\pi(uv) = f(u) \cdot f(v)$ for every edge $ uv \in E(G)$, where
\(f(u) \cdot f(v)\) is the scalar product of \(f(u)\) and \(f(v)\) over \(\mathbb{F}_2^t\).
Usually, we use the bold letter $\mathbf{u}$ to represent $f(u)$.
We use $\mathbf{0}$~(resp. $\mathbf{1}$) to represent vectors in $\mathbb{F}_2^t$ whose coordinates are all $0$~(resp. $1$).
We say a vector $\mathbf{u} \in \mathbb{F}_2^t$ is odd~(resp.~even), if $\mathbf{u \cdot 1} $ is one~(resp.~zero), i.e., $\mathbf{u}$ has an odd~(resp.~even) number of $1$s.

The inversion diameter has a close relation with vector assignment as given in the following proposition.

\begin{proposition}[\cite{havet2024diameter}]\label{prop: diam vector}
For every graph $G$ and every positive integer $t$, the following are equivalent.
\begin{enumerate}
    \item $\diam(\mathcal{I}(G)) \le t$.
    \item For every labelling $\pi$, there exists a $t$-dim vector assignment of $G$ respecting the labelling $\pi$.
\end{enumerate}
\end{proposition}

The treewidth of a graph $G$, denoted by $\tw(G)$ can be defined in many ways.
Here we give a definition of treewidth from the perspective of $k$-trees.

\begin{definition}\label{def: k tree}
A graph $G$ is a \emph{$k$-tree} if
\begin{enumerate}
    \item it is a $k$-clique, or
    \item there exists a vertex $v$ such that $N(v)$ is a $k$-clique, and $G-v$ is a k-tree.
\end{enumerate}
\end{definition}

We say a graph is a \emph{partial $k$-tree} if it is a subgraph of a $k$-tree.
It is known that a graph $G$ is a partial $k$-tree if and only if the treewidth of $G$ is at most $k$~\cite{scheffler1989baumweite, van1990graph}.

Let $\mathcal{L}(\mathbf{v}_1, \ldots, \mathbf{v}_k)$ denote the linear space spanned by $\mathbf{v}_1, \ldots, \mathbf{v}_k$.
For two vectors $\mathbf{v}$ and $\mathbf{u}$ in $\mathbb{F}_2^t$, we write $\mathbf{v \perp u}$ if $\mathbf{v \cdot u} = 0$.
For a vector $\mathbf{v}\in \mathbb{F}_2^t$ and a linear space $\mathbf{U}$ in $\mathbb{F}_2^t$, we write $\mathbf{v \perp U}$ if $\mathbf{v \perp u}$ for every $ \mathbf{u \in U}$. The orthogonal complementary space of $\mathbf{U}$ is $\mathbf{U}^{\perp} = \{ \mathbf{v} \mid \mathbf{v \perp U}\}$. For any positive integer $k$, we write $[k] = \{1,2,\ldots, k\}$.

\begin{definition}
    We say the vectors $\mathbf{v}_1, \ldots, \mathbf{v}_k$ are orthogonal if $\mathbf{v}_i \perp \mathbf{v}_j $ for all $i, j \in [k]$ with $i \neq j$.
    We say they are self-orthogonal if $\mathbf{v}_i \perp \mathbf{v}_j$ for all $i, j \in [k]$, that is, they are orthogonal and every vector is even.
\end{definition}

\begin{definition}
    A linear space $\mathbf{U}$ is self-orthogonal if $\mathbf{U \subseteq U^{\perp}}$.
\end{definition}

Let $\mathbf{U}$ be a self-orthogonal linear space. Then $\mathbf{U}$ is orthogonal and every vector in $\mathbf{U}$ is even.
It is easy to verify that $\mathbf{U}$ is self-orthogonal if and only if it has self-orthogonal base vectors.

For a linear space $\mathbf{U}$ and a vector $\mathbf{v}$, denote by $\mathbf{v + U}$ the set $\{ \mathbf{v+u} \mid \mathbf{u \in U}\}$ and denote by $\mathcal{L}(\mathbf{U,v})$ the space spanned by $\mathbf{v}$ and a basis of $\mathbf{U}$, that is the summation space of $\mathbf{U}$ and $\mathcal{L}(\mathbf{v})$.


\section{Proof of Theorem~\ref{thm: main}}\label{sec: main}

For  $k \ge 1$, we define a sequence of graphs $G_i^{(k)}$ respecting a fixed labelling $\pi_i^{(k)}$.
First, let $G_0^{(k)}$ be a $k$-clique respecting an arbitrary labelling $\pi_0^{(k)}$.
For convenience, we define $V(G^{(k)}_{-1}) = \emptyset$.
Then, we recursively construct $G_i^{(k)}$ as follows:

(i) for each $k$-clique with vertices $v_1,\ldots,v_k$ in $G_{i-1}^{(k)}$ and each $\mathbf{x} ={(x_1,\ldots,x_k)}^T\in \mathbb{F}_2^{k}$, we add  a new vertex $u$ such that $uv_j\in E(G_i^{(k)})$ and  $\pi_i^{(k)}(u v_j) = x_j$, for all $1\le j\le k$;

(ii) $\pi^{(k)}_{i} |_{G^{(k)}_{i-1}} = \pi^{(k)}_{i-1}$.

Since $|\mathbb{F}_2^{k}|=2^k$, we add $2^k$ new vertices for each $k$-clique in $G_{i-1}^{(k)}$. 
Observe that every $p$-clique~($p < k$) in $G_{i}^{(k)}$ must be contained in a $k$-clique in $G_{i}^{(k)}$.
By Definition~\ref{def: k tree}, $G_m^{(k)}$ is a $k$-tree for every $m$, that is, of treewidth at most $k$.
Since $\pi^{(k)}_{n} |_{G^{(k)}_{m}} = \pi^{(k)}_{m}$ when $n > m$, we may use $\pi^{(k)}$ to denote the labelling of $G^{(k)}_{m}$ for every $m$. For every vertex $v \in V(G^{(k)}_{m})$ with $m\ge 1$, there exists a unique $n$ such that $v \in V(G^{(k)}_n) - V(G^{(k)}_{n-1})$.
We say $n$ is the \emph{level} of $v$, denoted by $l(v) = n$.
For a vertex set $S \subseteq G^{(k)}_{m}$ with $m\ge 1$, the level of $S$ is defined to be the maximum level of a vertex in $S$ and it is denoted by $l(S)$, that is, $l(S) = \max_{v \in S} \{l(v)\}$.
Clearly, if $v$ is a vertex in $G_{m}^{(k)}$, then $l(v) \le m$.
Similarly, if $C$ is a vertex set in $G_{m}^{(k)}$, then $l(C) \le m$.


Note that if $H$ is a subgraph of $G$, then $\diam(\mathcal{I}(H)) \le \diam(\mathcal{I}(G))$.
So ${(\diam(\mathcal{I}(G_m^{(k)})))}_{m \ge 0}$ is an increasing sequence with upper bound $2k$ by Theorem~\ref{thm: havet}.

Let $\lambda^{(k)} = \lim_{m \rightarrow +\infty} \diam(\mathcal{I}(G_m^{(k)}))$.
Then $\lambda^{(k)} \le 2k$.
We will show that $\lambda^{(k)} = 2k$, that is, $G_m^{(k)}$ is of inversion diameter $2k$ when $m$ is sufficiently large.

Next we suppose that $\lambda^{(k)} \le 2k-1$. Then for every $m$, $G_m^{(k)}$  has a $(2k-1)$-dim vector assignment respecting the labelling $\pi^{(k)}$ by Proposition~\ref{prop: diam vector}. Thus for each $v\in V(G_m^{(k)})$, there is a vector $\mathbf{v}\in \mathbb{F}_2^{2k-1}$ corresponding to it.
The following lemmas show the properties of the vectors assigned to $k$-cliques in $G^{(k)}_{m}$.

\begin{lemma}\label{lemma: clique independence}
    If there is a $k$-clique of level $m$ with vertices $v_1, \ldots, v_k$ in $G^{(k)}_{m+1}$, then $\mathbf{v}_1, \ldots, \mathbf{v}_k$ are linearly independent.
\end{lemma}

\begin{proof}
    Otherwise, without loss of generality, assume $\mathbf{v}_1 = \sum_{i=2}^{k} c_i \mathbf{v}_i$ where $c_i \in \mathbb{F}_2$ for all $ 2 \le i \le k$. By the construction,
    there exists a vertex $u \in V(G^{(k)}_{m+1})$ which is connected to $v_2,\ldots, v_k$ with edges labelled by $1$, and to $v_1$ with an edge labelled by $\sum_{i=2}^{k}c_i + 1$.
    Therefore,
    \begin{equation*}
    \begin{aligned}
       \sum_{i=2}^{k}c_i + 1   & =  \pi^{(k)}(uv_1) = \mathbf{u} \cdot \mathbf{v}_1 = \mathbf{u} \cdot \sum_{i=2}^k c_i \mathbf{v}_i = \sum_{i=2}^k c_i \mathbf{u} \cdot \mathbf{v}_i  = \sum_{i=2}^k c_i\pi^{(k)}(uv_i) = \sum_{i=2}^k c_i,
    \end{aligned}
    \end{equation*}
    a contradiction.
\end{proof}

\begin{lemma}\label{lemma: independent or sum}
    If there is a $k$-clique  in $G^{(k)}_{m+2}$ of level $m$ with vertices $v_1, \ldots, v_k$, and $u$ is a vertex of level $m+1$ adjacent to all ${(v_i)}_{1 \le  i \le k}$,
    then either $\mathbf{v}_1, \ldots, \mathbf{v}_k, \mathbf{u}$ are linearly independent, or $\mathbf{u} = \sum_{i=1}^{k}\mathbf{v}_i$.
\end{lemma}

\begin{proof}
    Firstly, by Lemma~\ref{lemma: clique independence}, $\mathbf{v}_1, \ldots, \mathbf{v}_k$ are linearly independent.
    Note that for every $1 \le j \le k$, $v_1, \ldots, v_{j-1}, v_{j+1}, \ldots, v_k, u$ is also a $k$-clique of level $m+1$ in $G^{(k)}_{m+2}$.
    Then by Lemma~\ref{lemma: clique independence}, for every $1 \le j \le k$, $\mathbf{v}_1, \ldots, \mathbf{v}_{j-1}, \mathbf{v}_{j+1}, \ldots, \mathbf{v}_k, \mathbf{u}$ are linearly independent. \
    Assume $\mathbf{u} = \sum_{i=1}^{k} c_i \mathbf{v}_i$ where $c_i \in \mathbb{F}_2$ for all $1 \le i \le k$. If $c_j = 0$ for some $j$, then it contradicts that $\mathbf{v}_1, \ldots, \mathbf{v}_{j-1}, \mathbf{v}_{j+1}, \ldots, \mathbf{v}_k, \mathbf{u}$ are linearly independent. Therefore, $\mathbf{u} = \sum_{i=1}^{k} \mathbf{v}_i$.
\end{proof}

\begin{lemma}\label{lemma: sol independence or sum}
   Let $v_1, \ldots, v_k$ be vertices of
     a $k$-clique of level $m$ in $G^{(k)}_{m+2}$ and  $\mathbf{A} = {(\mathbf{v}_1, \ldots, \mathbf{v}_k)}^T$. Then for every $\mathbf{b} \in \mathbb{F}_2^k$, $\mathbf{Ax=b}$ has a solution $\mathbf{y}$ such that either $\mathbf{v}_1, \ldots, \mathbf{v}_k, \mathbf{y}$ are linearly independent, or $\mathbf{y} = \sum_{i=1}^{k} \mathbf{v}_i$.
\end{lemma}

\begin{proof} 
    Let $\mathbf{b}={(b_1,\ldots,b_k)}^T$. 
    By construction, there exists a vertex $u \in V(G^{(k)}_{m+1})$ of level $m+1$ connecting ${(v_i)}_{1 \le i \le k}$ such that $\pi^{(k)}(uv_i) = b_i$ for all $1 \le i \le k$. Then we have $\mathbf{Au=b}$.
    By Lemma~\ref{lemma: independent or sum}, either $\mathbf{v}_1, \ldots, \mathbf{v}_k, \mathbf{u}$ are linearly independent, or $\mathbf{u} = \sum_{i=1}^{k} \mathbf{v}_i$.
\end{proof}



The above actually work for arbitrary $\lambda^{(k)}$, while the following lemmas need the assumption $\lambda^{(k)} \le 2k-1$.

\begin{lemma}\label{lemma: 3}
         Let $v_1, \ldots, v_k$ be vertices of
     a $k$-clique of level $m$ in $G^{(k)}_{m+2}$ and  $\mathbf{A} = {(\mathbf{v}_1, \ldots, \mathbf{v}_k)}^T$.
    Then $\mathbf{Ax=0}$ has a solution $\mathbf{y}$ such that $\mathbf{v}_1, \ldots, \mathbf{v}_k, \mathbf{y}$ are linearly independent.
\end{lemma}

\begin{proof}
    We prove it by contradiction. By Lemma~\ref{lemma: clique independence}, $\mathbf{v}_1, \ldots, \mathbf{v}_k$ are linearly independent.
    Let $\mathbf{U}$ be the solution space of $\mathbf{Ax=0}$. Suppose $\mathbf{U}$ is a subspace of $\mathcal{L}(\mathbf{v}_1, \ldots, \mathbf{v}_k)$, otherwise we can pick $\mathbf{y}$ from $\mathbf{U} - \mathcal{L}(\mathbf{v}_1, \ldots, \mathbf{v}_k)$ and then $\mathbf{v}_1, \ldots, \mathbf{v}_k, \mathbf{y}$ are linearly independent.
    Since $\mathbf{A}$ is a $k\times (2k-1)$ matrix,
    $\dim(\mathbf{U}) =  (2k-1)- k = k-1$.
    By setting $\mathbf{b = 0}$ in Lemma~\ref{lemma: sol independence or sum}, we have $\sum_{i=1}^{k} \mathbf{v}_i \in \mathbf{U}$.

    For each $j\in [k]$, the solution set of $\mathbf{Ax=A}\mathbf{{}v}_j$ is in $\mathbf{v}_j + \mathbf{U} \subseteq \mathcal{L}(\mathbf{v}_1, \ldots, \mathbf{v}_k)$.
    Then any solution of $\mathbf{Ax=A}\mathbf{v}_j$ cannot be independent from $\mathbf{v}_1, \ldots, \mathbf{v}_k$.
    By Lemma~\ref{lemma: sol independence or sum}, there is a solution $\mathbf{y}$ such that either $\mathbf{v}_1, \ldots, \mathbf{v}_k, \mathbf{y}$ are linearly independent, or $\mathbf{y} = \sum_{i=1}^k \mathbf{v}_i$. 
    Since every solution is in $\mathbf{v}_j+ \mathbf{U} \subseteq \mathcal{L}(\mathbf{v}_1, \dots, \mathbf{v}_k)$, the first outcome does not occur, and so $\sum_{i=1}^k \mathbf{v}_i \in \mathbf{v}_j+ \mathbf{U}$.
    Therefore, $\mathbf{v}_j \in \mathbf{U}$ for every $1 \le j \le k$, which contradicts that $\dim(\mathbf{U}) = k-1$ because $\mathbf{v}_1, \ldots, \mathbf{v}_k$ are linearly independent.
\end{proof}

\begin{definition}\label{D2}
    Let $C$ be a $p$-clique of $G^{(k)}_{m}$ for some $m$. $C$ is called a \emph{bad} clique if $\dim(\textbf{V}_C \cap \textbf{V}_C^{\perp})\ge p-1$, where
    $\textbf{V}_C=\mathcal{L}(\{\mathbf{v}\mid v\in C\})$ and $p\ge 1$.
    \end{definition}

Note that a single vertex is always a bad $1$-clique.
If $\lambda^{(k)} \le  2k-1$, ``large'' bad cliques will finally cause contradictions.
The following lemma is the main part of our proof which states that we can find a ``large'' bad clique when $m$ is sufficiently large.

\begin{lemma}\label{lemma: 0 sol independence}
        If there exists a bad $p$-clique of level $m$ in $G^{(k)}_{m+k+2}$ with $p < k$,
    then there exists a bad clique in $G^{(k)}_{m+k+2}$ of size at least $p+1$.
\end{lemma}

\begin{proof}
    We prove it by contradiction. Suppose the $p$-clique $C_1$ with vertices $v_1,v_2,\ldots, v_p$ of level $m$ is the largest bad clique in $G^{(k)}_{m+k+2}$, where $p<k$.  Then $\dim(\textbf{V}_{C_1})=p$ by Lemma~\ref{lemma: clique independence}.
    Let $\textbf{U}=\textbf{V}_{C_1}\cap \textbf{V}_{C_1}^{\perp}$. Then  $\dim(\textbf{U})\ge p-1$ by Definition~\ref{D2}. For every $\mathbf{u}_1, \mathbf{u}_2\in \textbf{U}$, we have $\mathbf{u}_1\perp \mathbf{u}_2$ which means $\textbf{U}$ is self-orthogonal.

    We first show that $\dim(\textbf{U})= p-1$. 
    Suppose otherwise $\dim(\textbf{U})= p$.
    Then $\textbf{U}=\textbf{V}_{C_1} \subseteq \textbf{V}_{C_1}^{\perp}$. Since $p<k$, by the construction of $G^{(k)}_{m+k+2}$,  there exists a vertex $u$ of level $m+1$ such that $uv_i\in E(G^{(k)}_{m+k+2})$ and $\pi^{(k)}(uv_i)=0$ for each $i\in[p]$. Let  $C_1' \coloneqq {C_1} \cup \{u\}$.
    By Lemma~\ref{lemma: clique independence}, we have that $\mathbf{u}, \mathbf{v}_1, \ldots, \mathbf{v}_p$ are linearly independent.
    By $\pi^{(k)}(uv_i)=0$ for each $i\in[p]$, we have that $\mathbf{u} \perp \textbf{V}_{C_1}$. 
    Then for every $\mathbf{w} \in \textbf{V}_{C_1}$, we have $\mathbf{u} \perp \mathbf{w}$. 
    We also have $\mathbf{w} \perp \textbf{V}_{C_1}$ since $\textbf{V}_{C_1} \subseteq \textbf{V}_{C_1}^{\perp}$.  Therefore, $\mathbf{w} \perp \textbf{V}_{C_1'}$ and by the arbitrariness of $\mathbf{w}$, we have that $\textbf{V}_{C_1} \subseteq \textbf{V}_{C_1'} \cap \textbf{V}_{C_1'}^{\perp}$ which implies $\dim(\textbf{V}_{C_1'} \cap \textbf{V}_{C_1'}^{\perp})\ge p$. Hence, $C_1'$ is a bad $(p+1)$-clique,  a contradiction with the maximality of ${C_1}$, and hence, $\dim(\textbf{U})= p-1$.
    In fact, we conclude that $\textbf{U}$ is a self-orthogonal $(p-1)$-dim subspace of $\mathbb{F}_2^{2k-1}$ and each vector in $\textbf{U}$ is even. 

    Since  $\dim(\textbf{U})= p-1$, there exists $i \in \{1,2,\ldots, p\}$ such that $\mathbf{v}_i \notin \mathbf{U}$, say $i=1$.
   Then $\mathcal{L}(\textbf{U},\textbf{v}_1)=\textbf{V}_{C_1}$.
   If $\mathbf{v}_1$ is even, then $\mathbf{v}_1 \perp \mathcal{L}(\textbf{U},\textbf{v}_1)$, which contradicts with $\mathbf{v}_1 \notin \textbf{U}$. Thus we have that $\textbf{v}_1$ is odd. 

\begin{claim}\label{C zero neighbor is odd}
    If $u$ is a vertex in $G^{(k)}_{m+k}$ such that $uv_i\in E(G^{(k)}_{m+k})$ and $\pi^{(k)}(uv_i)=0$ for each $i\in[p]$, then $\textbf{u}$ is odd.
\end{claim}

\begin{proof}[of Claim~\ref{C zero neighbor is odd}]
Suppose $\textbf{u}$ is even. Then $\textbf{u}\perp \mathcal{L}(\textbf{u},\textbf{V}_{C_1})$. We have $\textbf{u}\in \textbf{V}_{C_1\cup\{u\}}\cap \textbf{V}_{C_1\cup\{u\}}^{\perp}$ and $\textbf{U}\subseteq \textbf{V}_{C_1\cup\{u\}}\cap \textbf{V}_{C_1\cup\{u\}}^{\perp}$. 
From Lemma~\ref{lemma: clique independence}, $\mathbf{v}_1, \ldots, \mathbf{v}_p, \mathbf{u}$ are linearly independent, so $\mathbf{u} \notin \textbf{V}_{C_1}$.
Recall that $\mathbf{U} = \mathbf{V}_{C_1} \cap \mathbf{V}_{C_1}^{\perp}$, then we have $\mathbf{u} \notin \mathbf{U}$.
Then $\dim(\textbf{V}_{C_1\cup\{u\}}\cap \textbf{V}_{C_1\cup\{u\}}^{\perp})\ge p$. Thus $ C_1\cup \{u\}$ is a bad $(p+1)$-clique, which contradicts the maximality of $C_1$. 
\end{proof}
\vspace{.2cm}

Fix a vertex $v_0$ of level $m+1$ such that $v_0v_i\in E(G^{(k)}_{m+k})$ and $\pi^{(k)}(v_0v_i)=0$ for each $i\in[p]$. Then $\mathbf{v}_0 \perp \mathbf{V}_{C_1}$.
By the construction of $G^{(k)}_{m+k}$, there exists a set of vertices $C_2=\{w_1,w_2,\dots, w_{k-p}\} \subseteq V(G^{(k)}_{m+k})$ satisfying the following:
\begin{enumerate}
    \item $\{v_0,v_1,\ldots,v_p,w_1,w_2,\ldots, w_{k-p}\}$ is a $(k+1)$-clique;
    \item $\pi^{(k)}(w_i w_j)=0$, for each $i,j\in[k-p], i\neq j$;
    \item and $\pi^{(k)}(v_i w_j)=\textbf{v}_i \cdot (\textbf{v}_1+\textbf{v}_0)$, for each $i\in[0,p],j\in[k-p]$.
\end{enumerate}

The existence of such $C_2$ is guaranteed because $C_1$ must be in some $k$-clique, and then we can pick $\mathbf{w}_1, \ldots, \mathbf{w}_{k-p}$ one by one according to the construction of $G^{(k)}_{m+k}$.

\begin{claim}\label{w even}
    $\textbf{w}_j$ is even for each $j\in[k-p]$.
\end{claim}

\begin{proof}[of Claim~\ref{w even}]
Suppose there is $j\in[k-p]$ such that $\textbf{w}_j$ is odd. Let  $\boldsymbol{\beta} \coloneqq \textbf{w}_j+\textbf{v}_1$. Recall that $\textbf{v}_1$ is odd.
Then 
\[
    \boldsymbol{\beta} \cdot \textbf{w}_j=\textbf{w}_j\cdot \textbf{w}_j+\textbf{v}_1\cdot \textbf{w}_j=\textbf{w}_j\cdot \textbf{w}_j+\textbf{v}_1\cdot (\textbf{v}_1+ \textbf{v}_0)=0.
\]
On the other hand, for every $v_i\in C_1$, 
\[
    \boldsymbol{\beta} \cdot \textbf{v}_i=\textbf{w}_j\cdot \textbf{v}_i+\textbf{v}_1\cdot \textbf{v}_i=\textbf{v}_i\cdot \textbf{v}_1+\textbf{v}_i\cdot \textbf{v}_0+\textbf{v}_1\cdot \textbf{v}_i=0.
\]
Hence $\boldsymbol{\beta} \perp \textbf{V}_{C_1\cup\{w_j\}}$. We have $\boldsymbol{\beta} \in \textbf{V}_{C_1\cup\{w_j\}}\cap \textbf{V}_{C_1\cup\{w_j\}}^{\perp}$ and $\textbf{U}\subseteq \textbf{V}_{C_1\cup\{w_j\}}\cap \textbf{V}_{C_1\cup\{w_j\}}^{\perp}$. From Lemma~\ref{lemma: clique independence} and $\dim(\textbf{U})= p-1$, we know $\textbf{w}_j \notin \textbf{U}$. Then $\dim(\textbf{V}_{C_1\cup\{w_j\}}\cap \textbf{V}_{C_1\cup\{w_j\}}^{\perp})\ge p$ which implies  $ C_1\cup \{w_j\}$ is a bad $(p+1)$-clique, a contradiction with the maximality of $C_1$.
\end{proof}

\vspace{.2cm}

\begin{table}
    \centering
    \begin{tabular}{|c|c|c|c|c|c|c|c|}
    \hline
         & $\textbf{v}_0$ & $\textbf{v}_1$  & $\textbf{v}_i$ & $\textbf{w}_{j_1}$ & $\textbf{w}_{j_2}$ & $\boldsymbol{\beta}_{j_1}$ &  $\boldsymbol{\beta}_{j_2}$\\
    \hline
        $\textbf{v}_0$ & 1 & 0 & 0 & 1 & 1 & 0 & 0\\
    \hline
        $\textbf{v}_1$ & 0 & 1 & $\textbf{v}_1\cdot \textbf{v}_i$ & 1 & 1 & 0 & 0\\
    \hline
        $\textbf{v}_i$ & 0 & $\textbf{v}_1\cdot \textbf{v}_i$ & $\textbf{v}_i\cdot \textbf{v}_i$ & $\textbf{v}_1\cdot \textbf{v}_i$ & $\textbf{v}_1\cdot \textbf{v}_i$ & 0 & 0\\
    \hline
        $\textbf{w}_{j_1}$ & 1 & 1 & $\textbf{v}_1\cdot \textbf{v}_i$ & 0 & 0 & 0 & 0\\
    \hline
        $\textbf{w}_{j_2}$ & 1 & 1 & $\textbf{v}_1\cdot \textbf{v}_i$ & 0 & 0 & 0 & 0\\
    \hline
        $\boldsymbol{\beta}_{j_1}$ & 0 & 0 & 0 & 0 & 0 & 0 & 0\\
    \hline
        $\boldsymbol{\beta}_{j_2}$ & 0 & 0 & 0 & 0 & 0 & 0 & 0\\
    \hline
    \end{tabular}
    \caption{The inner products for $i\in[p]$ and $j_1,j_2\in[k-p]$}\label{tab}
\end{table}

Now we complete the proof of Lemma~\ref{lemma: 0 sol independence}. For each $j\in[k-p]$, let $\boldsymbol{\beta}_j=\textbf{v}_0+\textbf{v}_1+\textbf{w}_j$. By Claim~\ref{C zero neighbor is odd},  $\textbf{v}_{0}$ is odd. By Claim~\ref{w even}, $\textbf{w}_j$ is even for each $j\in[k-p]$.
 It is not difficult to show that  $\boldsymbol{\beta}_j\perp \textbf{v}_{0}$, $\boldsymbol{\beta}_j\perp \textbf{V}_{C_1}$ and $\boldsymbol{\beta}_j\perp \textbf{V}_{C_2}$. Check Table~\ref{tab} for the inner products between the vectors that we are working with.
Let $C_3=C_1\cup C_2$ and $\mathbf{W}=\textbf{V}_{C_3}=\textbf{V}_{C_1}+\textbf{V}_{C_2}$. Then  $\dim(\mathbf{W})=k$ from Lemma~\ref{lemma: clique independence}. Since $\mathbf{W} \subseteq  \mathbb{F}_2^{2k-1}$, we have  $\dim(\mathbf{W} ^{\perp})=k-1$. Let $\mathbf{W} ^{\prime}=\mathcal{L}(\textbf{U},\boldsymbol{\beta}_1,\dots,\boldsymbol{\beta}_{k-p})$. Since $\textbf{U}\perp \mathbf{W}$ and $\boldsymbol{\beta}_j\perp \mathbf{W}$ for each $j\in[k-p]$, we have  $\mathbf{W}^{\prime}\subseteq \mathbf{W}^{\perp}$. Note that $\textbf{V}_{C_1},\textbf{V}_{C_2}\subseteq \mathcal{L}(\textbf{v}_0,\textbf{v}_1,\mathbf{W}^{\prime})$. We have $\mathbf{W} \subseteq \mathcal{L}(\textbf{v}_0,\textbf{v}_1,W^{\prime})$ which implies  $\dim(\mathbf{W}^{\prime})\ge k-2$. If $\textbf{v}_0\notin \mathbf{W}$, then  $\dim(\mathbf{W}^{\prime})\ge k-1$ which implies $\mathbf{W}^{\perp}=\mathbf{W}^{\prime}$. Since $\textbf{v}_0\perp \mathbf{W}^{\prime}$, we have $\textbf{v}_0\perp \mathbf{W}^{\perp}$, a contradiction with $\textbf{v}_0\notin \mathbf{W}$. Hence $\textbf{v}_0\in \mathbf{W}$.
By Lemma~\ref{lemma: independent or sum}, we have $ \textbf{v}_0+\cdots+\textbf{v}_p+\textbf{w}_1+\cdots+\textbf{w}_{k-p}=0 $.

Since $\textbf{v}_0\in \mathbf{W}$, we have $\boldsymbol{\beta}_j\in \mathbf{W}$ for each $j$ which implies $\mathbf{W}^{\prime}\subseteq \mathbf{W}$. So  $\mathbf{W}^{\prime}\subseteq \mathbf{W}\cap \mathbf{W}^{\perp}$.
If  $\dim(\mathbf{W}^{\prime})\ge k-1$, then $C_3$ is a bad $k$-clique,  a contradiction with the maximality of $C_1$. Hence $\dim(\mathbf{W}^{\prime})=k-2$. Let $\boldsymbol{\alpha} \in \mathbf{W}^{\perp}\backslash \mathbf{W}^{\prime}$ such that $\mathbf{W}^{\perp}=\mathcal{L}(\mathbf{W}^{\prime},\boldsymbol{\alpha})$. By the construction of $G^{(k)}_{m+k}$, there exists  a vertex $x$ connecting to all vertices of $C_3$ such that $\pi^{(k)}(xy) = 0$ for each $y \in C_3$. Then $\textbf{x}\in \mathbf{W}^{\perp}$.
From Claim~\ref{C zero neighbor is odd}, $\textbf{x}$ is odd. 
Since $\textbf{U}$ is a self-orthogonal subspace and $\boldsymbol{\beta}_i\bot \boldsymbol{\beta}_j$ for every $i,j\in [k-p]$ (see Table 1), we have that
the vectors in $\mathbf{W}^{\prime}$ are all even.  By $\textbf{x}\in \mathbf{W}^{\perp}=\mathcal{L}(\mathbf{W}^{\prime},\boldsymbol{\alpha})$ and $\textbf{x}$ being odd, we have $\boldsymbol{\alpha}$ is odd. 
Let $C^*=\{v_0,\ldots,v_p,w_1,\ldots,w_{k-p-1}\}$ (if $p=k-1$, then let $C^*=\{v_0,\ldots,v_p\}$). Then there is 
 $x^*$ connecting to all vertices of $C^*$ such that $\pi^{(k)}(x^* y) =\textbf{v}_0\cdot \textbf{y}$ for each $y\in C^*$. 
 
 Recall that $\mathbf{W} = \mathcal{L}(\mathbf{v}_1, \ldots, \mathbf{v}_p, \mathbf{w}_1, \ldots, \mathbf{w}_{k-p})$ and $\mathbf{V}_{C^*} = \mathcal{L}(\mathbf{v}_0, \ldots, \mathbf{v}_p, \mathbf{w}_1, \ldots, \mathbf{w}_{k-p-1})$. 
 Since $ \textbf{v}_0+\cdots+\textbf{v}_p+\textbf{w}_1+\cdots+\textbf{w}_{k-p}=0 $, we have that $\mathbf{W}=\textbf{V}_{C^*}$.
  Then $\textbf{x}^*\in \textbf{v}_0 + \mathbf{W}^{\perp}$. Since $\pi^{(k)}(x^* v_i) =\textbf{v}_0\cdot \textbf{v}_i=0$ for each $i\in [k]$,
  from Claim~\ref{C zero neighbor is odd}, $\textbf{x}^*$ is odd. Note that $\textbf{x}^*\in \textbf{v}_0+\mathcal{L}(\boldsymbol{\alpha},\mathbf{W}^{\prime})$. Since all vectors in $\mathbf{W}^{\prime}$ are  even and $\textbf{v}_0,\boldsymbol{\alpha}$ are odd, we have $\textbf{x}^*\in \textbf{v}_0+\mathbf{W}^{\prime}\subseteq \mathbf{W}$.
  From Lemma~\ref{lemma: independent or sum}, $\textbf{x}^*= \textbf{v}_0+\cdots+\textbf{v}_p+\textbf{w}_1+\cdots+\textbf{w}_{k-p-1}=\textbf{w}_{k-p}$. Since $\textbf{x}^{*}\cdot \mathbf{v}_1=0\neq 1=\textbf{w}_{k-p}\cdot \mathbf{v}_1$, we derive a contradiction.
\end{proof}

With the help of the above lemmas, we can derive a contradiction when $\lambda^{(k)} \le  2k-1$ and hence, $\lambda^{(k)} = 2k$.

\begin{theorem}\label{thm: lambda}
    $\lambda^{(k)} = 2k$.
\end{theorem}

\begin{proof}
    Suppose $\lambda^{(k)} \le 2k-1$. By Lemma~\ref{lemma: 0 sol independence}, the largest bad clique $C_0$ in $G^{(k)}_{k(k+2)}$ is of size $k$. Then $\dim(\textbf{V}_{C_0}^{\perp}\cap \textbf{V}_{C_0})\ge k-1$. By Lemma~\ref{lemma: clique independence}, $\dim(\textbf{V}_{C_0})=k$ and then $\dim(\textbf{V}_{C_0}^{\perp})=k-1$ by $\textbf{V}_{C_0}\subseteq \mathbb{F}_2^{2k-1}$.
    We have $\textbf{V}_{C_0}^{\perp}\subseteq \textbf{V}_{C_0}$ by checking the dimensions. 
    Then let $\mathbf{v}_1, \ldots, \mathbf{v}_k$ be a basis of $\textbf{V}_{C_0}$. 
    Then every solution $\mathbf{x}$ of $\mathbf{Ax} = \mathbf{0}$ is in $\textbf{V}_{C_0}^{\perp} \subseteq \textbf{V}_{C_0}$.
    Thus, $\mathbf{v}_1, \ldots, \mathbf{v}_k, \mathbf{x}$ can not be linearly independent, which contradicts Lemma~\ref{lemma: 3}.
\end{proof}

Now we can give the proof of our main Theorem.

\begin{proof}[of Theorem~\ref{thm: main}]
For every $k \ge 1$, we have $\lambda^{(k)} = 2k$ by Theorem~\ref{thm: lambda}. Then there exists $M_k$ such that for every $m \ge M_k$, $\diam(\mathcal{I}(G^{(k)}_m)) = 2k$. Thus, for all $m \ge M_k$, the graphs $G^{(k)}_m$ are the desired graphs of treewidth at most $k$ and inversion diameter $2k$.
\end{proof}
\begin{figure}[t]
	\centering         
	\includegraphics[width=0.5\linewidth]{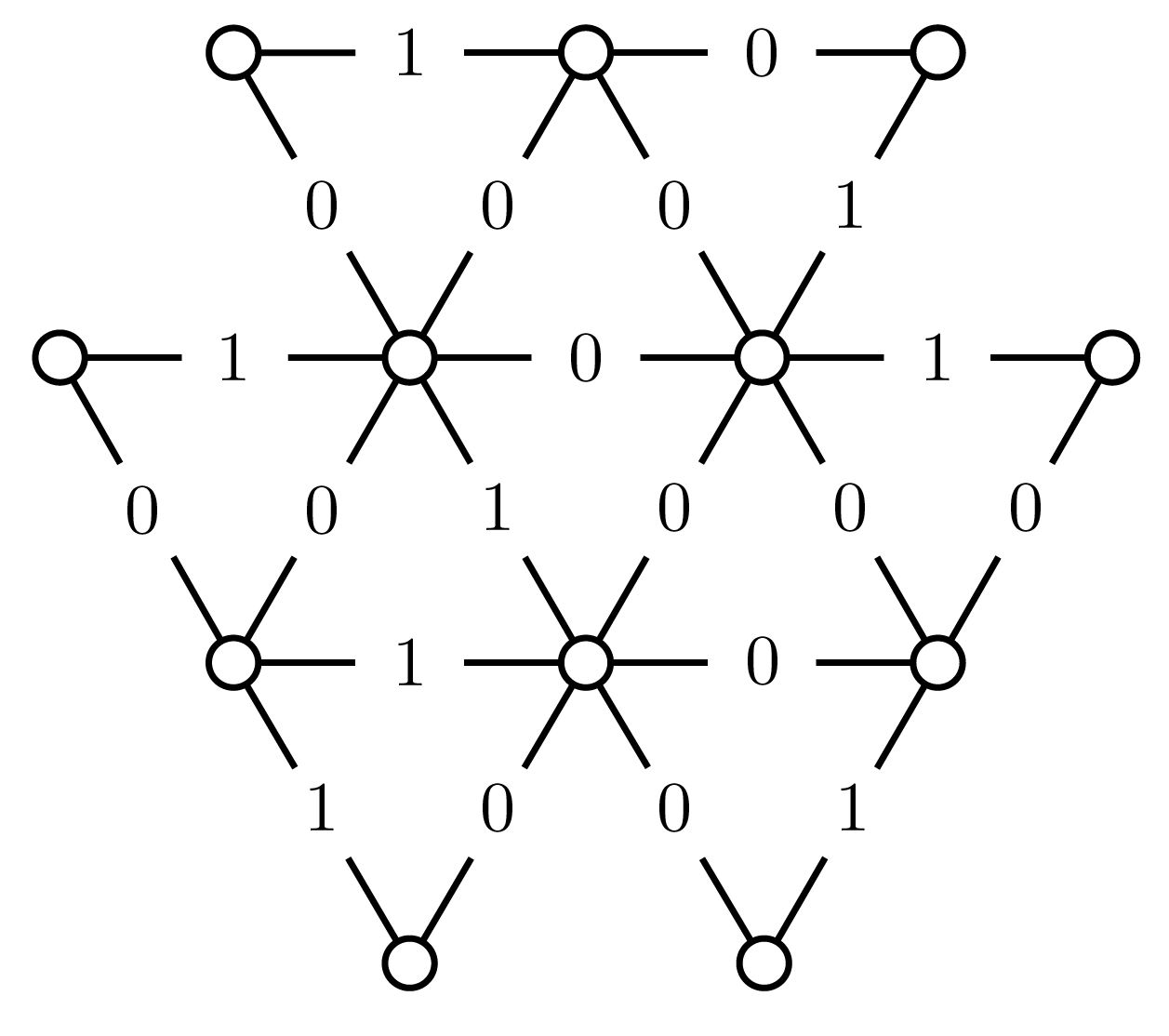}
	\caption{An example of outer-planar graph with labelled edges of inversion diameter $4$ verified by computer.}\label{fig: outerplanar}
\end{figure}

Note that every outer-planar graph is of treewidth $2$ and hence has inversion diameter at most $4$ by Lemma~\ref{thm: havet}.
We construct an outer-planar graph with inversion diameter $4$ verified by computer as Figure~\ref{fig: outerplanar}.
The idea is to construct an outerplanar graph as ``dense'' as possible, and the labelling is searched by computer.
The code is available on \href{https://github.com/handsome12138/InversionDiameter}{GitHub}.\footnote{https://github.com/handsome12138/InversionDiameter}
Therefore, the upper bound $\diam(\mathcal{I}(G)) \le 4$ for every outer-planar graph $G$ is tight.

\section{Proof of Theorem~\ref{thm: Delta 3 diam}}\label{sec: delta 3}

In this section, we intend to give the proof of Theorem~\ref{thm: Delta 3 diam}.

From the definition, if $\diam(\mathcal{I}(G)) = k$, then for every graph $G'$ obtained by removing a vertex from $G$, we have $\diam(\mathcal{I}(G')) \ge k-1$.
In other words, removing one vertex can decrease the inversion diameter by at most $1$.
Let $G$ be a graph. We say $G$ is \emph{$4$-diameter-critical} if $\diam(\mathcal{I}(G)) = 4$ and for every proper subgraph $G'$, $\diam(\mathcal{I}(G')) \le 3$.
Clearly, a $4$-diameter-critical graph is connected.
If $G$ is $4$-diameter-critical, by Proposition~\ref{prop: diam vector}, there exists a labelling $\pi$ such that there is no $3$-dim vector assignment of $G$ respecting $\pi$. We call such a labelling $\pi$ a \emph{bad labelling}.

Let $G$ be a $4$-diameter-critical graph respecting a bad labelling $\pi$ and
$H$ a non-empty induced subgraph of $G$. 
Denote by $N_G(H) = \{ v \in V(G)-V(H) \mid \exists u \in V(H), uv \in E(G) \}$ the neighbors of $H$ in $G-H$.
By the definition of a $4$-diameter-critical graph, $G-H$ admits a $3$-dim vector assignment $f: V(G-H) \rightarrow \mathbb{F}_2^3$ respecting $\pi|_{G-H}$.
For a vertex $v \in N_G(H)$, define $\mathcal{A}_{H,f}(v) = \{ \mathbf{v } \in \mathbb{F}_2^3 \mid \mathbf{v} \cdot f(u) = \pi(uv), \text{~for every~} uv \in E(G-H) \}$. Note that $f(v) \in \mathcal{A}_{H,f}(v)$.
Here $\mathcal{A}_{H,f}(v)$ is the set of all possible vectors that can be assigned to $v$ while keeping the vector assignment valid on $G-H$.

Let $H$ be a fixed induced subgraph of $G$ and $f$ a fixed $3$-dim vector assignment of $G-H$ respecting $\pi|_{G-H}$.
An \emph{available boundary family} is a family of sets ${(\mathcal{B}_f(v))}_{v \in N_G(H)}$ satisfying the following properties.
\begin{enumerate}
    \item $f(v) \in \mathcal{B}_f(v) \subseteq \mathcal{A}_{H,f}(v)$, and
    \item $\{ v \in N_G(H) \mid |\mathcal{B}_f(v)| \ge 2 \}$ is an independent set in $G-H$.
\end{enumerate}
When there is no ambiguity, we may ignore the subscript $f$ in $\mathcal{B}_f$.

The following lemma states that if we already have a vector assignment of $G-H$, then we can reassign the vectors $v \in N_G(H)$ using an available boundary family and the result is also a valid vector assignment.

\begin{lemma}\label{lemma: available boundary sets}
Let $H$ be an induced subgraph of a $4$-diameter-critical graph $G$ respecting a bad labelling $\pi$.
Let $f$ be a $3$-dim vector assignment of $G-H$ with $\pi|_{G-H}$ and ${(\mathcal{B}_f(v))}_{v \in N_G(H)}$ an available boundary family.
Then every $3$-dim vector assignment $g$ of $G-H$ satisfying
\begin{enumerate}
    \item $g(v) = f(v)$, $\forall v \in V(G-H) - N_G(H)$, and
    \item $g(v) \in \mathcal{B}_f(v) $, $\forall v \in N_G(H)$,
\end{enumerate}
is a $3$-dim vector assignment of $G-H$ with $\pi|_{G-H}$.
\end{lemma}

\begin{proof}
    We only need to verify that $g(v)\cdot g(u) = \pi(uv)$ for all $uv \in E(G-H)$.
    Note that $A \coloneqq \{ v \in V(G-H) \mid g(v) \neq f(v) \} \subseteq \{ v \in N_G(H) \mid |\mathcal{B}_f(v)| \ge 2 \}$ from the definition.
    Then $\{ v \in V(G-H) \mid g(v) \neq f(v) \}$ is an independent set.
    Since we already have $f(v)\cdot f(u) = \pi(uv)$ for all $uv \in E(G-H)$ and $\{ v \in V(G-H) \mid g(v) \neq f(v) \}$ is an independent set, we now only need to verify that $g(v)\cdot g(u) = \pi(uv)$ for all $uv \in E(G-H)$ satisfying $u \in A$ and $v \notin A$.
    Since $g(u) \in \mathcal{B}_f(u)$ and $g(v) = f(v)$, we have $g(v)\cdot g(u) = \pi(uv)$ by the definition of $\mathcal{B}_f(u)$.
\end{proof}

We say that $H$ is \emph{reducible} respecting the labelling $\pi$ if there exists a $3$-dim vector assignment $f$ of $G-H$ respecting $\pi|_{G-H}$, and an available boundary family ${(\mathcal{B}_f(v))}_{v \in N_G(H)}$ and a $3$-dim vector assignment $g$ on $G[V(H) \cup N_G(H)]$ with  $\pi|_{G[V(H) \cup N_G(H)]}$ such that $g(v) \in \mathcal{B}_f(v)$ for every $v \in N_G(H)$. The following lemma states that there is no reducible subgraph of a $4$-diameter-critical graph.

\begin{lemma}\label{lemma: reducible}
Let $G$ be a $4$-diameter-critical graph respecting a bad labelling $\pi$.
Then there is no reducible induced subgraph of $G$.
\end{lemma}

\begin{proof}
    Suppose  $H$ is an induced reducible subgraph of $G$.
    Then $G-H$ admits a $3$-dim vector assignment $f$ respecting the labelling $\pi|_{G-H}$, an available boundary family ${(\mathcal{B}_f(v))}_{v \in N_G(H)}$ and a $3$-dim vector assignment $g$ on $G[V(H) \cup N_G(H)]$ such that $g(v) \in \mathcal{B}_f(v)$ for every $v \in N_G(H)$.
    Define a function $h: V(G) \rightarrow \mathbb{F}_2^3$ by letting $h(v) = f(v)$ for every $v \in V(G-H) - N_G(H)$ and $h(v) = g(v)$ for every $v \in N_G(H)\cup V(H)$.
    By the definition, $h|_{G[V(H) \cup N_G(H)]}$ is a $3$-dim vector assignment respecting the labelling $\pi|_{G[V(H) \cup N_G(H)]}$.
    By Lemma~\ref{lemma: available boundary sets}, $h|_{G-H}$ is a $3$-dim vector assignment of $G-H$ respecting the labelling $\pi|_{G-H}$. Since there is no edge between $V(G-H) - N_G(H)$ and $V(H)$, $h$ is a $3$-dim vector assignment of $G$ with $\pi$, a contradiction.
\end{proof}

In the following, we are going to find certain reducible structures in $4$-diameter-critical graphs.

\begin{lemma}\label{lemma: reducible all zero adjacent edges}
Let $G$ be a $4$-diameter-critical graph respecting a bad labelling $\pi$.
For every vertex $v \in V(G)$, at least one edge adjacent to $v$ is labelled $1$ by $\pi$.
\end{lemma}

\begin{proof}
    Suppose there exists a vertex $v \in V(G)$ such that $\pi(uv) = 0$ for all $u \in N_G(v)$.
    Let $G' = G - v$. Then  $G'$ admits a $3$-dim vector assignment $f$ with  $\pi|_{G'}$.
    Let $f(v) = \mathbf{0} \in \mathbb{F}_2^3$. Then it is not difficult to  verify that $f$ is a $3$-dim vector assignment of $G$ with  $\pi$, a contradiction.
\end{proof}

\begin{lemma}\label{lemma: nonzero, degree <= 2}
    Let $G$ be a graph respecting a labelling $\pi$.
    If $G$ admits a $3$-dim vector assignment with $\pi$, then there exists a $3$-dim vector assignment $f$ with $\pi$ such that $f(v) \neq \mathbf{0}$ for every vertex $v \in V(G)$ of degree at most $2$.
\end{lemma}

\begin{proof}
    Let $f$ be the $3$-dim vector assignment of $G$ with $\pi$ which minimizes $n_f = |\{ v \in V(G) \mid f(v) = \mathbf{0}, d_G(v) \le 2\}|$.
    Suppose otherwise $n_f > 0$. Let $w\in \{ v \in V(G) \mid f(v) = \mathbf{0}, d_G(v) \le 2\}$ 
    and $\mathcal{F}(w) =  \{ \mathbf{w } \in \mathbb{F}_2^3 \mid \mathbf{w \cdot} f(u) = \pi(uw), \forall uw \in E(G) \}$. Then $|\mathcal{F}(w)| \ge 2$ since $d_G(w) \le 2$.
    Choose $\mathbf{w }\in \mathcal{F}(w)-\{\mathbf{ 0}\}$ and define a function $g: V(G)\rightarrow \mathbb{F}_2^3$ by letting $g(v) = f(v)$ for every $v \in V(G) - \{w\}$ and $g(w) = \mathbf{w}$. It is easy to verify that $g$ is a $3$-dim vector assignment of $G$ with $\pi$, but $n_g < n_f$, a contradiction.
\end{proof}

\begin{lemma}\label{lemma: reducible 3-regular}
Let $G$ be a $4$-diameter-critical graph of maximum degree $3$ respecting a bad labelling $\pi$.
Then $G$ is 3-regular.
\end{lemma}

\begin{proof} 
    Suppose there exists a vertex $v \in V(G)$ such that $d(v)=1$. Let $uv\in E(G)$.
    Then by Lemma~\ref{lemma: reducible all zero adjacent edges}, $\pi(uv) = 1$.
    Let $V(H) = \{v\}$. Then $N_G(H) = \{u\}$. By hypothesis, $G-v$ admits a $3$-dim vector assignment $f$ with $\pi|_{G-v}$.
    Since $d_{G-v}(u)\le 2$,  $|\mathcal{A}_{H,f}(u)| \ge 2$.
    Let $\mathcal{B}_f(u) = \mathcal{A}_{H,f}(u)$. Then ${(\mathcal{B}_f(u))}_{u \in N_G(H)}$ is an available boundary family. Let  $g(u)\in \mathcal{B}(u)-\{\mathbf{0 }\}$. We can choose $g(v)\in \mathbb{F}_2^3$ such that $g(v)\cdot g(u)=1$. Then
     $H$ is reducible, a contradiction with Lemma~\ref{lemma: reducible}.

    Suppose there exists a vertex $v \in V(G)$ such that $d(v)=2$. Let $N_H(v)=\{u_1,u_2\}$.
    By Lemma~\ref{lemma: reducible all zero adjacent edges}, without loss of generality, assume $\pi(vu_1) = 1$.
    Let $V(H) = \{v\}$. Then $N_G(H) = \{u_1,u_2\}$. By hypothesis and Lemma~\ref{lemma: nonzero, degree <= 2}, $G-v$ admits a $3$-dim vector assignment $f$  with $\pi|_{G-v}$ such that $f(u_1), f(u_2) \neq \mathbf{0}$.
    Let $\mathcal{B}(u_1) = \{f(u_1)\}$ and $\mathcal{B}(u_2) =   \mathcal{A}_{H,f}(u_2)$. Then ${(\mathcal{B}(u_i))}_{i=1,2}$ is an available boundary family. 
    Since $d_{G-v}(u_2) \le 2$, we have $|\mathcal{B}(u_2)| \ge 2$. Let $g(u_1)=f(u_1)$. 
    If $\pi(vu_2) = 1$ (resp. $\pi(vu_2) = 0$), choose $g(u_2) \in \mathcal{B}(u_2)-\{\mathbf{ 0}\} $ (resp.  $g(u_2) \in \mathcal{B}(u_2)-\{f(u_1)\} $). 
    It is easy to verify in either case that there exists $g(v) \in \mathbb{F}_2^3$ such that $g(v) \cdot g(u_i) = \pi(vu_i)$ for $i=1,2$, so $H$ is reducible, a contradiction with Lemma~\ref{lemma: reducible}.
\end{proof}

\begin{figure}[t]
	\begin{minipage}{0.49\linewidth}
		\vspace{3pt}
		\centerline{\includegraphics[width=0.7\textwidth]{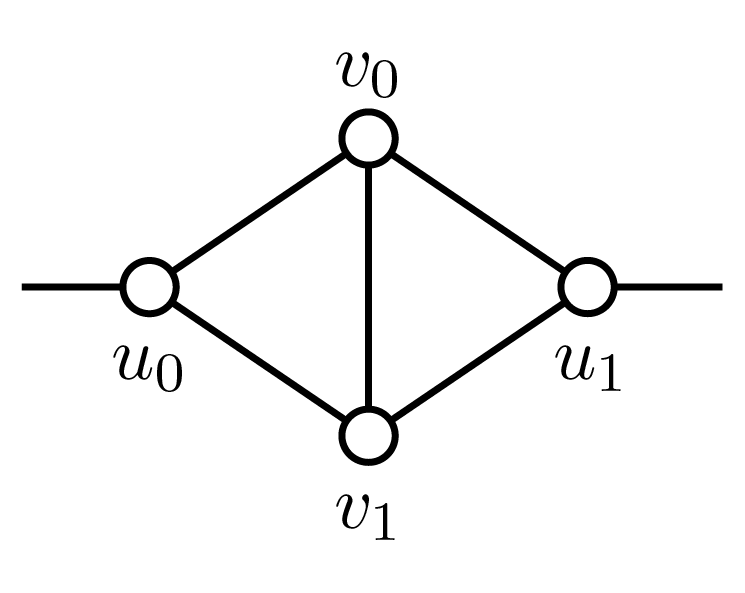}}
		\caption{$K_4^-$ in $G$.}\label{fig: K4-}
	\end{minipage}
	\begin{minipage}{0.49\linewidth}
		\vspace{3pt}
		\centerline{\includegraphics[width=0.7\textwidth]{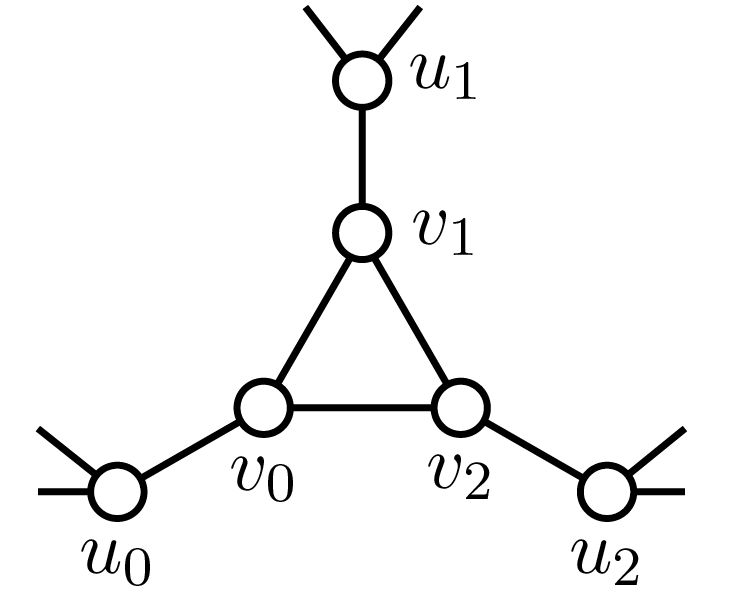}}
		\caption{Triangle in $G$.}\label{fig: C3}
	\end{minipage}
\end{figure}

\begin{lemma}\label{lemma: reducible K4-}
Let $G$ be a $4$-diameter-critical 3-regular graph respecting a bad labelling $\pi$.
There is no induced $K_4^-$ in $G$, where $K_4^-$ is the graph obtained by deleting an edge in $K_4$.
\end{lemma}

\begin{proof}
    Suppose there exists a $K_4^-$ in $G$ with vertex set $\{v_0,v_1,u_0,u_1\}$ and $u_0u_1 \notin E(G)$~(see Figure~\ref{fig: K4-}).

    Let $H = G[\{v_0,v_1\}]$. Then $N_G(H) = \{u_0, u_1\}$. By hypothesis and Lemma~\ref{lemma: nonzero, degree <= 2}, $G-H$ admits a $3$-dim vector assignment $f$  with $\pi|_{G-H}$ such that $f(u_0), f(u_1) \neq \mathbf{0}$.
    Let $\mathcal{B}(u_i) = \mathcal{A}_{H,f}(u_i)$ for $i=0,1$. Then ${(\mathcal{B}(u_i))}_{i=0,1}$ is an available boundary family.
    We have the following properties:
    \begin{enumerate}
        \item For each $i\in \{0,1\}$, $|\mathcal{B}(u_i) | \ge 4$ as $d_{G-H}(u_i) = 1$.
        \item For each $i\in \{0,1\}$, if $\pi(v_0u_i) = \pi(v_1u_i) = 0$, then $\mathbf{0} \notin \mathcal{B}(u_i)$ by Lemma~\ref{lemma: reducible all zero adjacent edges}.
        \item For each $i\in \{0,1\}$, at least one edge in $\{v_0v_{1}, v_i u_0, v_i u_1\}$ is labelled one by Lemma~\ref{lemma: reducible all zero adjacent edges}.
    \end{enumerate}

    With the above properties, we claim that $H$ is reducible.
    The claim is proved by using a computer to enumerate all available boundary families with above properties.
    The source codes can be found on \href{https://github.com/handsome12138/InversionDiameter}{GitHub}.
    From this, we derive a contradiction with Lemma~\ref{lemma: reducible}.
\end{proof}

\begin{lemma}\label{lemma: reducible C3}
Let $G$ be a $4$-diameter-critical 3-regular graph respecting a bad labelling $\pi$.
Then there is no triangle in $G$.
\end{lemma}

\begin{proof}
    Suppose  there exists a triangle with vertices $\{v_0,v_1,v_2\}$ and, for $i=0,1,2$, let $u_i$ be the neighbor of $v_i$~(see Figure~\ref{fig: C3}). By Lemma~\ref{lemma: reducible K4-}, $u_0,u_1,u_2$ are either distinct vertices, or $u_0=u_1=u_2$.
    If $u_0=u_1=u_2$, then $G=K_4$ by $G$ being 3-regular.
    However, it was shown in~\cite{havet2024diameter} that $\diam(\mathcal{I}(K_4)) = 3$, which contradicts the fact that $G$ is $4$-diameter-critical.
    Hence, we conclude that $u_0,u_1,u_2$ are distinct vertices.
    Let $V(H)=\{v_0,v_1,v_2\}$. Then $N_G(H) = \{u_0, u_1,u_2\}$. By hypothesis and Lemma~\ref{lemma: nonzero, degree <= 2}, $G-H$ admits a $3$-dim vector assignment $f$ with $\pi|_{G-H}$ such that
     $f(u_i) \neq \mathbf{0}, i=0,1,2$. 
     By relabelling as necessary, we can assume that $u_0$ satisfies the property: if $f(u_1) = f(u_2)$, then  $f(u_0) =f(u_1) = f(u_2)$.
    Let $\mathcal{B}(u_0) = \mathcal{A}_{H,f}(u_0)$ and $\mathcal{B}(u_i) = \{f(u_i)\}, i= 1,2$.
    Now we have the following properties:
    \begin{enumerate}
        \item $|\mathcal{B}(u_0)| \ge 2$ as $d_{G-H}(u_0) = 2$.
        \item For each $i=0,1,2$, at least one edge adjacent to $v_i$ is labelled one by Lemma~\ref{lemma: reducible all zero adjacent edges}.
        \item If $\pi(u_0v_0) = 0$, then $\mathbf{0} \notin \mathcal{B}(u_0)$, also by Lemma~\ref{lemma: reducible all zero adjacent edges}.
        \item If $f(u_1) = f(u_2)$, then $f(u_1) = f(u_0) \in \mathcal{B}(u_0)$.
    \end{enumerate}

    With the above properties, we claim that $H$ is reducible, which is again proved with the help of a computer.
    The source code can be found on \href{https://github.com/handsome12138/InversionDiameter}{GitHub}.
    Therefore, we derive a contradiction with Lemma~\ref{lemma: reducible}.
\end{proof}

\begin{figure}[t]
	\begin{minipage}{0.49\linewidth}
		\vspace{3pt}
		\centerline{\includegraphics[width=0.7\textwidth]{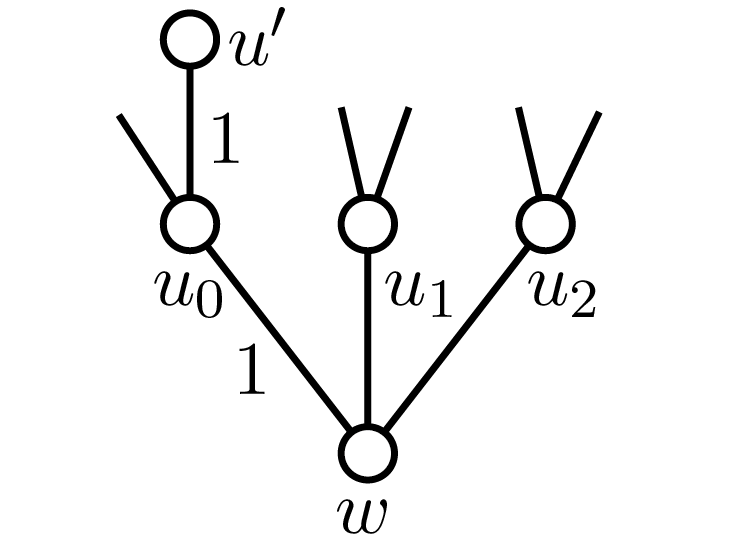}}
		\caption{$P_3$ with edges labelled one in $G$.}\label{fig: P3}
	\end{minipage}
	\begin{minipage}{0.49\linewidth}
		\vspace{3pt}
		\centerline{\includegraphics[width=0.7\textwidth]{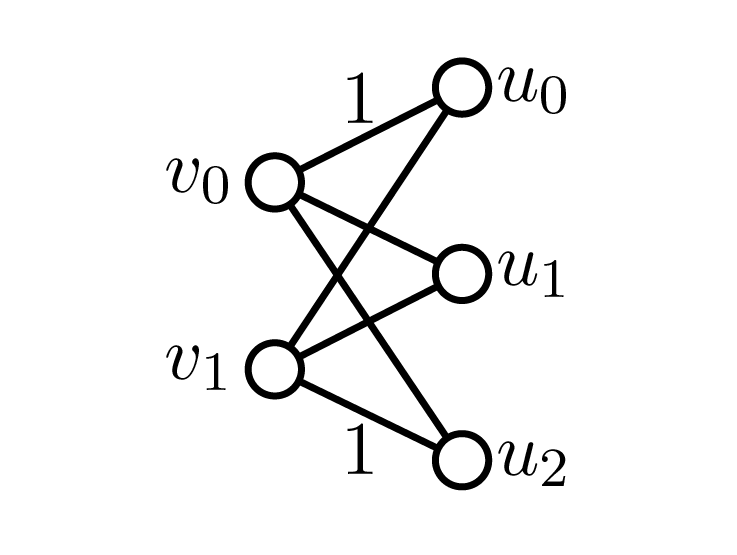}}
		\caption{$K_{2,3}$ in $G$.}\label{fig: K23}
	\end{minipage}
\end{figure}

\begin{lemma}\label{lemma: reducible P3 labelled 1}
Let $G$ be a $4$-diameter-critical 3-regular graph respecting a bad labelling $\pi$.
Then there is no $P_3$ with two edges labelled one in $G$.
\end{lemma}

\begin{proof}
    Suppose  there exists a path $wu_0u'$ such that $\pi(wu_0) = \pi(u_0u') = 1$.
    By Lemma~\ref{lemma: reducible C3}, $u'w\notin E(G)$. 
    Let $u_1, u_2$ be the neighbors of $w$~(see Figure~\ref{fig: P3}).
    By Lemma~\ref{lemma: reducible C3}, $\{u_0,u_1,u_2\}$ is an independent set.
     Let $V(H)=\{w\}$. Then $N_G(H) = \{u_0, u_1,u_2\}$. By hypothesis, $G-H$ admits a $3$-dim vector assignment $f$ with $\pi|_{G-H}$.
     Let $\mathcal{B}(u_i) = \mathcal{A}_{H,f}(u_i), i=0,1,2$. Then  ${(\mathcal{B}(u_i))}_{i=0,1,2}$ is an available boundary family.
     We have the following properties:
     \begin{enumerate}
        \item For each $i\in \{0,1,2\}$, $|\mathcal{B}(u_i)| \ge 2$ as $d_{G-H}(u_i) = 2$.
         \item $\mathbf{0} \notin \mathcal{B}(u_0)$, because $\pi(u_0u') = 1$.
         \item For each $i =1,2$, if $\pi(wu_i) = 0$, then $\mathbf{0} \notin \mathcal{B}(u_i)$ by Lemma~\ref{lemma: reducible all zero adjacent edges}.
     \end{enumerate}
     With the above properties, we claim that $H$ is reducible which is checked using a computer~(\href{https://github.com/handsome12138/InversionDiameter}{GitHub}).
     From this, we derive a contradiction with Lemma~\ref{lemma: reducible}.
\end{proof}

\begin{lemma}\label{lemma: reducible K23}
Let $G$ be a $4$-diameter-critical 3-regular graph respecting a bad labelling $\pi$.
Then there is no $K_{2,3}$ in $G$.
\end{lemma}

\begin{proof}
    Suppose  there exists a $K_{2,3}$ with vertices ${\{v_i\}}_{i=0,1} \cup {\{u_i\}}_{i=0,1,2}$ and $u_i v_j \in E(G)$ for every $i=0,1$ and $j=0,1,2$~(see Figure~\ref{fig: K23}).
    By Lemmas~\ref{lemma: reducible all zero adjacent edges} and~\ref{lemma: reducible P3 labelled 1}, without loss of generality, we can assume $\pi(v_0u_0) = \pi(v_1u_2) = 1$ and other edges in $K_{2,3}$ are labelled zero.
    By Lemma~\ref{lemma: reducible C3}, ${\{u_i\}}_{i=0,1,2}$ is an independent set.
    Let $H = G[\{v_0,v_1\}]$. Then $N_G(H) = \{u_0, u_1,u_2\}$. By hypothesis, $G-H$ admits a $3$-dim vector assignment $f$.
    Let $\mathcal{B}(u_i) = \mathcal{A}_{H,f}(u_i), i=0,1,2$. Then  ${(\mathcal{B}(u_i))}_{i=0,1,2}$ is an available boundary family.
         We have the following properties:
     \begin{enumerate}
        \item For each $i\in\{0,1,2\}$, $|\mathcal{B}(u_i)| \ge 4$ as $d_{G-H}(u_i) = 1$.
         \item By Lemma~\ref{lemma: reducible all zero adjacent edges}, $\mathbf{0} \notin \mathcal{B}(u_1)$ and $\mathbf{0} \in \mathcal{B}(u_i)$ for $i=0,2$.
     \end{enumerate}
     With the above properties, we claim that $H$ is reducible which is checked using a computer~(\href{https://github.com/handsome12138/InversionDiameter}{GitHub}).
     From this, we derive a contradiction with Lemma~\ref{lemma: reducible}.
\end{proof}

\begin{figure}[t]
	\begin{minipage}{0.49\linewidth}
		\vspace{3pt}
		\centerline{\includegraphics[width=0.7\textwidth]{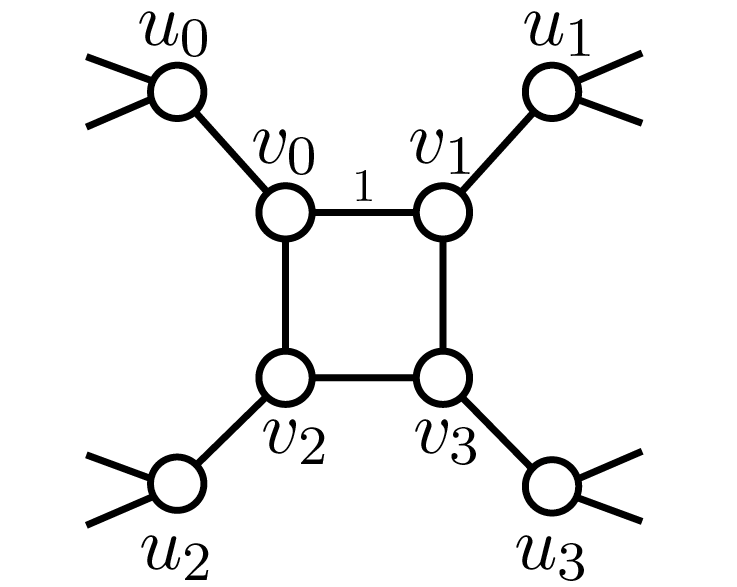}}
		\caption{$C_4$ with at least one edge labelled one in $G$.}\label{fig: C4}
	\end{minipage}
	\begin{minipage}{0.49\linewidth}
		\vspace{3pt}
		\centerline{\includegraphics[width=0.7\textwidth]{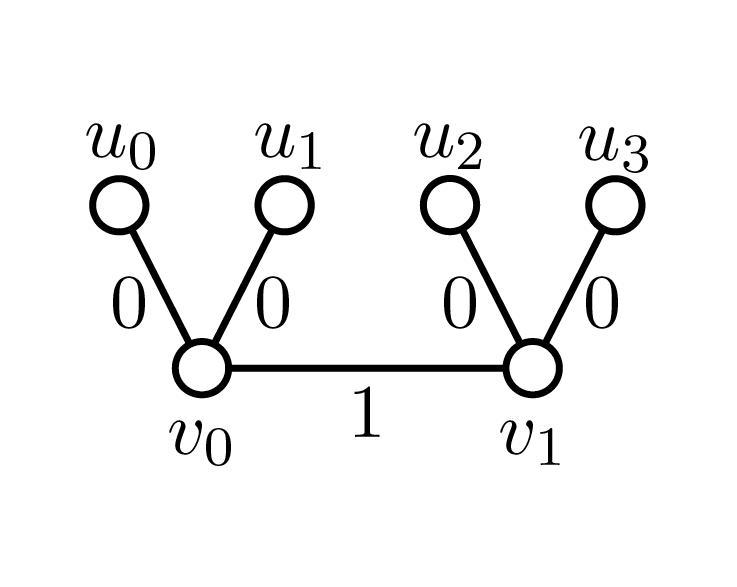}}
		\caption{Edge labelled one in $G$.}\label{fig: bridge}
	\end{minipage}
\end{figure}

\begin{lemma}\label{lemma: reducible C4 with label 1}
Let $G$ be a $4$-diameter-critical 3-regular graph respecting a bad labelling $\pi$.
Then there is no $C_4$ with at least one edge labelled one in $G$.
\end{lemma}

\begin{proof}  
    Suppose  there exists a copy of $C_4$ with vertices ${\{v_i\}}_{i=0,1,2,3}$ and with $\pi(v_0v_1) = 1$.
    Let $u_i$ be the neighbor of $v_i$ for $i=0,1,2,3$~(see Figure~\ref{fig: C4}).
    By Lemmas~\ref{lemma: reducible C3} and~\ref{lemma: reducible K23}, ${\{u_i\}}_{i=0,1,2,3}$ are distinct vertices.
    By Lemma~\ref{lemma: reducible P3 labelled 1}, $\pi(v_0v_2) = \pi(v_1v_3) = 0$.
    Let $H=C_4$. Then $N_G(H) = \{u_0, u_1,u_2, u_3\}$. 
    By hypothesis and Lemma~\ref{lemma: nonzero, degree <= 2}, $G-H$ admits a $3$-dim vector assignment $f$  with $\pi|_{G-H}$ such that  $f(u_i) \neq \mathbf{0}, i=0,1,2,3$.

    Let us first consider the case $\pi(v_2v_3) = 0$. In this case, let $\mathcal{B}(u_i) = \{f(u_i)\}, i=0,1,2,3$. Then ${(\mathcal{B}(u_i))}_{i=0,1,2,3}$ is an available boundary family. We claim $H$ is reducible with ${(\mathcal{B}(u_i))}_{i=0,1,2,3}$, which is proved by computer search~(\href{https://github.com/handsome12138/InversionDiameter}{GitHub}), and gives a contradiction with Lemma~\ref{lemma: reducible}.

    Now suppose $\pi(v_2v_3) = 1$. Then $\pi(v_i u_i) = 0, i=0,1,2,3$ by Lemma~\ref{lemma: reducible P3 labelled 1}.
    Since $G$ is 3-regular, there exists $t \in \{1,2,3\}$ such that $u_0u_t \notin E(G)$. Let $\mathcal{B}(u_i) = \mathcal{A}_{H,f}(u_i)$ for $i=0,t$ and $\mathcal{B}(u_i) = \{f(u_i)\}$ for other $i$.
    Then ${(\mathcal{B}(u_i))}_{i=0,1,2,3}$ is an available boundary family. Since $\pi(v_i u_i) = 0$ for each $i$, we have $\mathbf{0} \notin \mathcal{B}(u_i)$ by Lemma~\ref{lemma: reducible all zero adjacent edges}. Moreover, for $i=0,t$, we have $|\mathcal{B}(u_i) | \ge 2$ since $d_{G-H}(u_i) = 2$.
    We claim that we can choose $\mathbf{u}_i \in \mathcal{B}(u_i)$ for every $i\in \{0,1,2,3\}$ such that  $\mathbf{u}_{j_0} = \mathbf{u}_{j_1}$ and $\mathbf{u}_{j_2} = \mathbf{u}_{j_3}$ do not occur, where  $\{j_0,j_1,j_2,j_3\} = \{0,1,2,3\}$. The claim can be proved easily.

    Define $g: V(G-H) \rightarrow \mathbb{F}_2^3$ by letting $g(u_i) = \mathbf{u}_i, i=0,1,2,3$ and $g(v) = f(v)$ for all other vertices $v$. By Lemma~\ref{lemma: available boundary sets}, $g$ is a $3$-dim vector assignment of $G-H$ with $\pi|_{G-H}$.

    Let $\mathcal{B}_g(u_i) = g(u_i), i=0,1,2,3$, then ${(\mathcal{B}_g(u_i))}_{i=0,1,2,3}$ is an available boundary family.
    We claim $H$ is reducible with ${(\mathcal{B}_g(u_i))}_{i=0,1,2,3}$, which is proved by computer search~(\href{https://github.com/handsome12138/InversionDiameter}{GitHub}), and gives a contradiction.
\end{proof}

Now we have plenty of forbidden structures in $G$, and we can finally prove Theorem~\ref{thm: Delta 3 diam}.

\begin{proof}[of Theorem~\ref{thm: Delta 3 diam}] By contradiction.
    Let $G$ be a counterexample with the minimum number of vertices and, amongst all such examples, the minimum number of edges.
    Then $G$ is $4$-diameter-critical. 
    Let $\pi$ be a bad labelling of $G$.
    Since $\Delta(G)\le 3$, $G$ is 3-regular by Lemma~\ref{lemma: reducible 3-regular}.

     By Lemma~\ref{lemma: reducible all zero adjacent edges}, at least one edge is labelled one in $G$.
    Pick an edge $v_0v_1 \in E(G)$ labelled one.
    Let $u_0,u_1$ be the neighbors of $v_0$ and $u_2,u_3$ be the neighbors of $v_1$~(see Figure~\ref{fig: bridge}).
    By Lemmas~\ref{lemma: reducible C3} and~\ref{lemma: reducible C4 with label 1}, ${\{u_i\}}_{i=0,1,2,3}$ is an independent set.
    Let $H = G[\{v_0,v_1\}]$. Then $N_G(H) = \{u_0, u_1,u_2, u_3\}$. By hypothesis, $G-H$ admits a $3$-dim vector assignment $f$.
    Let $\mathcal{B}(u_i) = \mathcal{A}_{H,f}(u_i)$ for $i=0,1,2,3$. Then ${(\mathcal{B}(u_i))}_{i=0,1,2,3}$ is an available boundary family. 
    Since $d_{G-H}(u_i) = 2$ for all $i=0,1,2,3$, we have $|\mathcal{B}(u_i)| \ge 2$ for all $i=0,1,2,3$.
    By Lemma~\ref{lemma: reducible all zero adjacent edges}, at least one edge adjacent to $u_i$ is labelled one for every $i=0,1,2,3$.
    We know $v_0u_0, v_0u_1, v_1u_2, v_1u_3$ are labelled zero.
    Thus, there is at least one edge adjacent to each $u_i$ in $G-H$ labelled one.
    Then by definition, $\mathbf{0} \notin \mathcal{B}(u_i)$ for every $i=0,1,2,3$.
    Then we claim that $H$ is reducible, which is proved using a computer by enumerating all possibilities for $\mathcal{B}(u_i)$ for each $i=1,2,3,4$~(\href{https://github.com/handsome12138/InversionDiameter}{GitHub}), and gives a contradiction with Lemma~\ref{lemma: reducible}.
\end{proof}


\acknowledgements
\label{sec:ack}
The authors would like to thank the referee for their careful reading and valuable comments which help to improve the presentation of this paper.
Y. Yang is supported by the Fundamental Research Funds for the Central University~(Grant 500423306) in China.
M. Lu is supported by the National Natural Science Foundation of China~(Grant 12571372).

\nocite{*}
\bibliographystyle{abbrvnat}
\bibliography{sample-dmtcs}

@article{alon2024invertibility,
  title={Invertibility of digraphs and tournaments},
  author={Alon, Noga and Powierski, Emil and Savery, Michael and Scott, Alex and Wilmer, Elizabeth},
  journal={SIAM Journal on Discrete Mathematics},
  volume={38},
  number={1},
  pages={327--347},
  year={2024},
  publisher={SIAM}
}

@article{aubian2025problems,
  title={Problems, Proofs, and Disproofs on the Inversion Number},
  author={Aubian, Guillaume and Havet, Fr{\'e}d{\'e}ric and H{\"o}rsch, Florian and Kingelhoefer, Felix and Nisse, Nicolas and Rambaud, Cl{\'e}ment and Vermande, Quentin},
  journal={The Electronic Journal of Combinatorics},
  volume={32},
  number={1},
  year={2025}
}

@article{bang2022inversion,
  title={On the inversion number of oriented graphs},
  author={Bang-Jensen, J{\o}rgen and da Silva, Jonas Costa Ferreira and Havet, Fr{\'e}d{\'e}ric},
  journal={Discrete Mathematics \& Theoretical Computer Science},
  volume={23},
  number={Special issues},
  year={2022},
  publisher={Episciences. org}
}

@article{belkhechine2010inversion,
  title={Inversion dans les tournois},
  author={Belkhechine, Houmem and Bouaziz, Moncef and Boudabbous, Imed and Pouzet, Maurice},
  journal={Comptes Rendus. Math{\'e}matique},
  volume={348},
  number={13-14},
  pages={703--707},
  year={2010}
}

@article{havet2024diameter,
     author = {Havet, Fr\'ed\'eric and H\"orsch, Florian and Rambaud, Cl\'ement},
     title = {Diameter of the inversion graph},
     journal = {Innovations in Graph Theory},
     pages = {49--88},
     year = {2026},
     publisher = {Stichting Innovations in Graph Theory},
     volume = {3},
     language = {en}
}

@book{scheffler1989baumweite,
  title={Die Baumweite von Graphen als ein Ma{\ss} f{\"u}r die Kompliziertheit algorithmischer Probleme},
  author={Scheffler, Petra},
  volume={4},
  year={1989},
  publisher={Akademie der Wissenschaften der DDR, Karl-Weierstrass-Institut f{\"u}r Mathematik}
}

@incollection{van1990graph,
  title={Graph algorithms},
  author={van Leeuwen, Jan},
  booktitle={Algorithms and complexity},
  pages={525--631},
  year={1990},
  publisher={Elsevier}
}
\label{sec:biblio}

\end{document}